\documentclass[a4paper,12pt]{amsart}

\usepackage[T1]{fontenc}
\usepackage{lmodern}
\usepackage{a4wide}
\usepackage{microtype}
\usepackage{amsmath,amssymb,amsthm,mathtools}
\usepackage[mathcal]{eucal}
\usepackage{enumitem}
\usepackage{url}
\usepackage{hyperref}
\usepackage{bookmark}

\hypersetup{
  colorlinks=true,
  linkcolor=[rgb]{0.10,0.10,0.45},
  citecolor=[rgb]{0.10,0.10,0.45},
  urlcolor=[rgb]{0.10,0.10,0.45},
  pdftitle={Maximal right ideals of B(E)},
  pdfauthor={Tomasz Kania and Niels Jakob Laustsen}
}

\numberwithin{equation}{section}

\theoremstyle{plain}
\newtheorem{maintheorem}{Theorem}

\newtheorem{theorem}{Theorem}[section]
\newtheorem{proposition}[theorem]{Proposition}
\newtheorem{corollary}[theorem]{Corollary}
\newtheorem{lemma}[theorem]{Lemma}
\newtheorem{question}[theorem]{Question}

\theoremstyle{definition}

\newtheorem{example}[theorem]{Example}
\newtheorem{remark}[theorem]{Remark}

\newcommand{\cA}{\mathcal A}
\newcommand{\cB}{\mathcal B}
\newcommand{\cF}{\mathcal F}
\newcommand{\cK}{\mathcal K}
\newcommand{\cS}{\mathcal S}
\newcommand{\C}{\mathbb C}
\newcommand{\N}{\mathbb N}
\newcommand{\id}{I}
\newcommand{\ran}{\operatorname{ran}}
\newcommand{\Lift}{\operatorname{Lift}}
\newcommand{\Com}{\operatorname{Com}}
\newcommand{\op}{\mathrm{op}}

\title[Maximal right ideals of $\cB(E)$]{Maximal right ideals of the
Banach algebra\\ of bounded operators on a Banach space}
\author[T.~Kania]{Tomasz Kania}
\address[T.~Kania]{Mathematical Institute\\Czech Academy of Sciences\\\v{Z}itn\'a 25 \\115 67 Praha 1\\Czech Republic  and  Institute of Mathematics and Computer Science\\ Jagiellonian University\\ {\L}ojasiewicza 6, 30-348 Krak\'{o}w, Poland
}
\email{kania@math.cas.cz, tomasz.marcin.kania@gmail.com}
\thanks{RVO: 67985840.}

\author[N.J.~Laustsen]{Niels Jakob Laustsen} 
\address[N.J.~Laustsen]{School of Mathematical Sciences, Charles Carter Building, Lancaster University, Lancaster LA1 4YX, United Kingdom}
\email{n.laustsen@lancaster.ac.uk}


\begin{document}

\begin{abstract}
We study finitely generated maximal right ideals of the
Banach algebra $\cB(E)$ of bounded operators on a complex Banach
space~$E$.  Every maximal right ideal is either fixed by a non-zero
functional or contains the ideal of finite-rank operators; when $E$ is
infinite-dimensional, each non-fixed maximal right ideal in fact
contains the ideal of inessential operators.

Using the elementary representation of finitely generated right ideals
as lifting ideals
$\Lift(T)=\{TU:U\in\cB(E,E^n)\}$, where $T\in\cB(E^n,E)$ for some $n\in\N$, we identify the exact operator-theoretic obstruction.
The ideal $\Lift(T)$ contains the finite-rank operators precisely when
$T$ is surjective, and it equals $\cB(E)$ precisely when $T$ is right
invertible.  If $T$ is surjective but not right invertible, then
$\Lift(T)$ is maximal exactly when the row operator $[T\ S]$ is right
invertible for every $S\in\cB(E)\setminus\Lift(T)$.

We apply this framework, together with duality, pullback,
lattice-theoretic and cardinality arguments, to obtain maximal right
ideals which are not finitely generated for large classes of Banach
spaces.  These include the following infinite-dimensional spaces:
reflexive spaces, separable spaces with an unconditional Schauder
decomposition into a countably infinite sequence of non-zero subspaces,
spaces containing a complemented copy of $\ell_1$, KB-spaces,
Lebesgue spaces $L_p(\mu)$ for $1\leqslant p<\infty$,
full Orlicz spaces with order-continuous norm, and scalar-plus-compact
spaces.  We obtain the stronger conclusion that
every finitely generated maximal right ideal is fixed for Hilbert
spaces, $\ell_1(\Gamma)$-spaces, reflexive spaces with the bounded
approximation property, and the mixed spaces
$\ell_1(\Gamma)\oplus H$ with $H$ a separable Hilbert space.
\end{abstract}

\maketitle

\section{Introduction}
\noindent
Let $A$ be a unital complex Banach algebra.  The Dales--\.{Z}elazko conjecture asserts that if every maximal left ideal of $A$ is finitely generated, then $A$ is finite-dimensional~\cite{DalesZelazko}.  Passing to the opposite algebra interchanges left and right ideals without changing finite generation, so the corresponding assertion for maximal right ideals is not a different conjecture.  The general problem remains open; White has recently given an account of the general state of affairs in his work on Beurling algebras~\cite{WhiteBeurling}. 

An important positive case was established by Blecher and the
first-named author, who proved in~\cite{BlecherKania} that a
$C^*$-algebra~$A$ is finite-dimensional if and only if every maximal
right ideal of~$A$ is finitely generated.  Thus
$C^*$-algebras satisfy the right-sided Dales--\.{Z}elazko conjecture.
A key ingredient in their argument is that every closed, 
finitely generated right ideal of a $C^*$-algebra is of the form $pA$
for some projection $p\in A$.  They also obtained an analogous
characterisation of Hilbert $C^*$-modules whose maximal right
submodules are all finitely generated.

In collaboration with Dales, Kochanek, and Koszmider~\cite{DKKKL}, we developed the left-sided theory systematically for $A=\cB(E)$ (the algebra of bounded operators on a Banach space~$E$), exploiting the additional geometric structure that this case offers, with further results obtained in~\cite{KLindiana}. 
The purpose of this note is to formulate the right-sided theory in a way that makes the remaining obstruction completely explicit and avoids replacing a~quotient map by a bounded linear right inverse without justification. 

We shall now summarise our main findings; we refer to Section~\ref{S:prelim} for details of any unexplained notation or conventions. 

Let $E$ be a non-zero Banach space, and take a non-zero functional $\lambda\in E^*$. A straight\-forward argument shows that the set
\begin{equation}\label{eq:fixedideal}
  \mathcal M^r_\lambda
  =\{T\in\cB(E):\lambda T=0\}  
\end{equation}
is a maximal right ideal of~$\cB(E)$ generated by an idempotent operator; see Proposition~\ref{prop:fixed} for details. We call the  maximal right ideals of this form \emph{fixed}.

Our first main result is the right-sided counterpart of \cite[Theorem~1.1]{DKKKL}.

\begin{maintheorem}[Right-sided dichotomy]\label{thm:A}
  Let~$\mathcal R$ be a maximal right ideal of $\cB(E)$ for some non-zero Banach space~$E$.  Exactly one of the following two alternatives holds:
\begin{enumerate}[label=\rm(\roman*)]
\item $\mathcal R=\mathcal M^r_\lambda$ for some $0\neq\lambda\in E^*$; or
\item $\mathcal R$ contains the ideal~$\cF(E)$ of finite-rank operators on~$E$.
\end{enumerate}
\end{maintheorem}

For $n\in\N$ and $T\in\cB(E^n,E)$, set
\begin{equation}\label{eq:Lift}
\Lift(T) = \{ TU : U\in\cB(E,E^n)\}\subseteq\cB(E).    
\end{equation}
Our second main result identifies the operator-theoretic content of this
representation.

\begin{maintheorem}[Lifting-ideal formulation]\label{thm:B}
Let $E$ be a Banach space.
\begin{enumerate}[label={\rm(\roman*)}]
\item\label{thm:B:i}
A subset $\mathcal R$ of $\cB(E)$ is a finitely generated right ideal if
and only if
\[
  \mathcal R=\Lift(T)
\]
for some $n\in\N$ and some $T\in\cB(E^n,E)$.
\end{enumerate}
Let $T\in\cB(E^n,E)$ for some $n\in\N$. Then:
\begin{enumerate}[label={\rm(\roman*)},resume]
\item\label{thm:B:ii}
$\cF(E)\subseteq\Lift(T)$ if and only if $T$ is surjective.
\item\label{thm:B:iii}
$\Lift(T)=\cB(E)$ if and only if $T$ is right invertible.
\item\label{thm:B:iv}
Suppose that $T$ is not right invertible. Then the right ideal $\Lift(T)$ is maximal if
and only if, for every $S\in\cB(E)\setminus\Lift(T)$, the row operator
\begin{equation}\label{thm:B:iv:row_op}
  [T\ S]\colon E^n\oplus E\longrightarrow E,
  \qquad [T\ S](x,y)=Tx+Sy,
\end{equation}
is right invertible.
\end{enumerate}
\end{maintheorem}

Thus a non-fixed, finitely generated maximal right ideal is not merely induced by a row operator; it is exactly a maximal lifting ideal associated with a bounded linear surjection $E^n\twoheadrightarrow E$ that is not right invertible.

\begin{maintheorem}\label{thm:C}
Let $E$ be a Banach space, and suppose that at least one of the following two conditions is satisfied:
\begin{enumerate}[label=\rm(\roman*)]
\item\label{thm:C:i} For every $n\in\N$, every surjective operator $T\in\cB(E^n,E)$ is right invertible.
\item\label{thm:C:ii} $E$ is reflexive and has the bounded approximation property.
\end{enumerate}
Then every finitely generated maximal right ideal of $\cB(E)$ is fixed.  If $E$ is infinite-dimensional, $\cB(E)$ therefore contains a maximal right ideal which is not finitely generated.
\end{maintheorem}

We note that condition~\ref{thm:C:i} holds when $E$ is isomorphic to a Hilbert space or to $\ell_1(\Gamma)$ for some index set~$\Gamma$; see Proposition~\ref{prop:examplesCi} for details. 

A different argument applies when $E$ has enough coordinate projections.

\begin{maintheorem}\label{thm:D}
Let $E$ be a separable Banach space with an unconditional Schauder decomposition into a countably infinite sequence of non-zero subspaces.  Then $\cB(E)$ contains $2^{\mathfrak{c}}$ maximal right ideals, and $2^{\mathfrak{c}}$ of them are not finitely generated, where $\mathfrak{c}=2^{\aleph_0}$.
\end{maintheorem}

We also record three complementary mechanisms.  First,
infinite-dimensional commutative quotients pull non-finitely-generated
maximal ideals back to the original algebra.  Secondly, reflexivity
transfers White's left-sided result to the right by the adjoint
anti-isomorphism.  Thirdly, on an infinite-dimensional
scalar-plus-inessential space, the inessential operators form the
unique non-fixed maximal right ideal; in the infinite-dimensional
scalar-plus-compact case, this ideal cannot be finitely generated.

No claim is made here that the general conjecture, or the case of $\cB(E)$ for arbitrary $E$, has been solved.  The exact unresolved geometric question is stated in Section~\ref{sec:remaining}.

\section{Preliminaries}\label{S:prelim}

\noindent All Banach spaces and Banach algebras are complex.  Algebraic finite generation is meant throughout; no closure is inserted into a generated ideal. The letter~$E$ always denotes a Banach space and~$E^*$ its topological dual. For $n\in\N$, we equip~$E^n$ with the $\ell_p$-norm for some $p\in[1,\infty]$; the precise choice of~$p$ does not matter because the $\ell_p$-norms are equivalent, and all our results are isomorphic in nature. For $1\leqslant j\leqslant n$, $\iota_j\colon E\to E^n$ and $\pi_j\colon E^n\to E$ denote the $j^{\text{th}}$ coordinate embedding and projection, respectively.

Let~$E$ and~$F$ be Banach spaces. We write~$\cB(E,F)$ for the space of bounded, linear operators $E\to F$, abbreviated $\cB(E)$ when $E=F$. We say that an operator $T\in\cB(E,F)$ is \emph{right invertible} if it has a bounded linear right inverse $F\to E$. It is an elementary fact that this is equivalent to~$T$ being surjective with its kernel $\ker T$ being complemented in~$E$.

Given $y\in F$ and $\lambda\in E^*$, the rank-one operator $y\otimes\lambda\in\cB(E,F)$ is defined by
\[
  (y\otimes\lambda)(x)=\lambda(x)y
  \qquad (x\in E).
\]
The symbols $\cA$, $\cF$, $\cK$ and $\cS$ denote, respectively, the operator ideals of approximable, finite-rank, compact and strictly singular operators.

We begin with a few basic facts which are useful for keeping the algebraic and topological assertions separate. They are well known; we include  their (short) proofs for ease of reference.

\begin{lemma}\label{lem:max-closed}
Every maximal right ideal of a unital Banach algebra is norm closed.
\end{lemma}

\begin{proof} Assume towards a contradiction that~$A$ is  a unital Banach algebra which contains a~maximal right ideal~$M$ that is not closed. Continuity of the algebra operations implies that its closure~$\overline{M}$ is also a right ideal, so $\overline{M}=A$ by maximality of~$M$.  Choose $m\in M$ such that $\|1-m\|<1$.  Then $m$ is invertible, and hence $1=mm^{-1}\in M$, a contradiction.
\end{proof}

\begin{proposition}\label{prop:opposite}
The left- and right-sided Dales--\.{Z}elazko conjectures are equivalent.  

More precisely, a subset $M$ of a unital Banach algebra~$A$ is a maximal right ideal of $A$ if and only if it is a maximal left ideal of $A^{\op}$, and
\[
  M=a_1A+\cdots+a_nA
\]
if and only if
\[
  M=A^{\op}a_1+\cdots+A^{\op}a_n
\]
in the opposite algebra.
\end{proposition}

\begin{proof}
The product in $A^{\op}$ is $(a,b)\mapsto ba$.  Hence $A^{\op}a_j=\{a_jb:b\in A\}=a_jA$.  The assertions about ideals, maximality and finite generation follow immediately.
\end{proof}

\begin{proposition}[Pullback principle]\label{prop:pullback}
Let $\pi\colon A\to B$ be a surjective homomorphism of unital algebras,
and let $M$ be a maximal right ideal of~$B$.  Then $\pi^{-1}(M)$ is a
maximal right ideal of~$A$.  If $\pi^{-1}(M)$ is finitely generated,
then $M$ is finitely generated.  Consequently, non-finite generation
pulls back along~$\pi$.
\end{proposition}

\begin{proof}
A surjective homomorphism between unital algebras is automatically
unital.  Let $\mathcal R$ be a right ideal properly containing
$\pi^{-1}(M)$, and choose $x\in\mathcal R\setminus\pi^{-1}(M)$.  Then
$\pi(x)\notin M$, so $\pi(\mathcal R)$ properly contains~$M$ and is
therefore equal to~$B$.  Choose $y\in\mathcal R$ with $\pi(y)=1_B$.
Since $1_A-y\in\ker\pi\subseteq\pi^{-1}(M)\subseteq\mathcal R$, we
have $1_A\in\mathcal R$, and hence $\mathcal R=A$.

If $\pi^{-1}(M)=a_1A+\cdots+a_nA$, then surjectivity gives
$M=\pi(a_1)B+\cdots+\pi(a_n)B$.
\end{proof}

For a Banach algebra~$A$, let $\Com(A)$ denote the closed two-sided
ideal generated by the commutators $ab-ba$.

\begin{corollary}\label{cor:comm-quotient}
For a unital Banach algebra~$A$, the following conditions are
equivalent:
\begin{enumerate}[label={\rm(\roman*)}]
\item $A$ has an infinite-dimensional commutative Banach-algebra
quotient;
\item the abelianisation $A/\Com(A)$ is infinite-dimensional.
\end{enumerate}
Whenever these conditions hold, $A$ has a maximal right ideal which is
not finitely generated.
\end{corollary}

\begin{proof}
The second condition plainly implies the first.  Conversely, suppose
that \mbox{$\rho\colon A\to B$} is a continuous surjective homomorphism onto
an infinite-dimensional commutative Banach algebra.  Then
$\Com(A)\subseteq\ker\rho$, so $\rho$ factors through a surjective
homomorphism $A/\Com(A)\to B$.  Hence $A/\Com(A)$ is
infinite-dimensional.

The commutative case of the Dales--\.{Z}elazko conjecture gives a
non-finitely-generated maximal ideal in $A/\Com(A)$, and
Proposition~\ref{prop:pullback} pulls it back to~$A$.
\end{proof}

\section{Fixed ideals and the dichotomy}

\begin{proposition}\label{prop:fixed}
Let $E$ be a Banach space. For each $0\neq\lambda\in E^*$, the set $\mathcal M^r_\lambda$ defined by~\eqref{eq:fixedideal} is a maximal right ideal which is generated by the idempotent operator $\id_E-y\otimes\lambda$ for any element $y\in E$ such that $\lambda(y)=1$. 
\end{proposition}

\begin{proof} The fact that $\lambda(y)=1$ implies that the operator $P=\id_E-y\otimes\lambda\in\cB(E)$ is idempotent and satisfies $\lambda P=0$, so $P\in\mathcal M^r_\lambda$. On the other hand, for each $T\in\mathcal M^r_\lambda$, we have
\[
  P T=T-y\otimes(\lambda T)=T.
\]
This shows that $\mathcal M^r_\lambda$ is the right ideal generated by~$P$. 

To verify the maximality of~$\mathcal M^r_\lambda$, suppose that~$\mathcal R$ is a right ideal which properly contains~$\mathcal M^r_\lambda$, and take $T\in \mathcal R\setminus\mathcal M^r_\lambda$.  Then we can find $x\in E$ such that $\lambda(Tx)=1$.  Consequently,
\[
  \lambda\bigl(\id_E-T(x\otimes\lambda)\bigr)=0,
\]
so $\id_E-T(x\otimes\lambda)\in\mathcal M^r_\lambda\subseteq\mathcal R$.  Since $T(x\otimes\lambda)\in\mathcal R$, we conclude that $\id_E\in\mathcal R$, and therefore $\mathcal R=\cB(E)$.
\end{proof}

\begin{proof}[Proof of Theorem~\ref{thm:A}]
Assume that $\cF(E)\nsubseteq \mathcal R$.  Since the rank-one operators span $\cF(E)$, we can find non-zero elements $x\in E$ and $\lambda\in E^*$ such that $x\otimes\lambda\notin \mathcal R$.  Maximality of~$\mathcal R$ gives an operator $T\in\cB(E)$ such that
\(
  \id_E-(x\otimes\lambda)T\in \mathcal R.
\)
Put $\mu=\lambda T\in E^*$.  Then $(x\otimes\lambda)T=x\otimes\mu$, and $\mu\neq0$, since otherwise $\id_E\in \mathcal R$.

Choose $y\in E$ with $\mu(y)=1$, and set $P=\id_E-y\otimes\mu$.  Since $\mu P=0$, we have
\[
P=(\id_E-x\otimes\mu)P=\bigl(\id_E-(x\otimes\lambda)T\bigr)P\in\mathcal R.
\]
Proposition~\ref{prop:fixed} gives
\[
  \mathcal M^r_\mu=P\cB(E)\subseteq \mathcal R,
\]
so $\mathcal R=\mathcal M^r_\mu$ by maximality of~$\mathcal M^r_\mu$.

The alternatives are mutually exclusive: if $0\neq\mu\in E^*$, choose $z\in E$ with $\mu(z)\neq0$. Then an operator of rank one with range $\C z$ does not belong to $\mathcal M^r_\mu$.
\end{proof}

\begin{corollary}\label{cor:fixed-implies-DZ}
Let $E$ be an infinite-dimensional Banach space.  If every finitely generated maximal right ideal of $\cB(E)$ is fixed, then $\cB(E)$ contains a maximal right ideal which is not finitely generated.
\end{corollary}

\begin{proof}
The proper two-sided ideal $\cF(E)$ of~$\cB(E)$ is contained in some maximal right ideal~$\mathcal R$.  This ideal is non-fixed by Theorem~\ref{thm:A}, and hence is not finitely generated by hypothesis.
\end{proof}

In analogy with the left-sided case, this dichotomy has a stronger form \cite[Corollary~4.1]{DKKKL}. Its proof carries over almost verbatim. Before giving the details, we recall the key notion and background results required.

An operator $T\in\cB(E,F)$ between Banach spaces~$E$ and~$F$ is \emph{inessential} if $\id_E - ST$ is a~Fredholm operator for every operator \mbox{$S\in\cB(F,E)$}. (We remark that this condition is left-right symmetric because $\id_E - ST$ is a Fredholm operator if and only if $\id_F - TS$ is a Fredholm operator.) We write  $\mathcal E(E,F)$ for the set of inessential operators from~$E$ to~$F$. This defines a closed operator ideal in the sense of Pietsch, who gave this definition~\cite{Pietsch}. It generalises Kleinecke's original definition~\cite{kl} of the ideal of inessential operators, which applied only when $E=F$, defining~$\mathcal E(E)$ as the pre\-image of the Jacobson radical~$\operatorname{rad}\bigl(\cB(E)/\cA(E)\bigr)$ of the quotient algebra~$\cB(E)/\cA(E)$. This was motivated by Yood's observation~\cite[p.~615]{yo} that this radical may be non-zero. 

\begin{corollary}[Strong right-sided dichotomy]\label{cor:strictly-singular}
Let $E$ be an infinite-dimensional Banach space. Then every non-fixed maximal right ideal of $\cB(E)$ contains the ideal $\mathcal E(E)$ of inessential operators.
\end{corollary}

\begin{proof}
  Let $\mathcal R$ be a non-fixed maximal right ideal of $\cB(E)$. Then $\cF(E)\subseteq \mathcal R$ by Theorem~\ref{thm:A}, and therefore $\cA(E)\subseteq \mathcal R$ by Lemma~\ref{lem:max-closed}. This implies that $\pi(\mathcal R)$ is a maximal right ideal of~$\cB(E)/\cA(E)$, where $\pi\colon\cB(E)\to\cB(E)/\cA(E)$ denotes the quotient homomorphism, so $\operatorname{rad}\bigl(\cB(E)/\cA(E)\bigr)\subseteq \pi(\mathcal R)$ by the standard characterization of the Jacobson radical of a~unital algebra as the intersection of its maximal right ideals. In conclusion, we have
  \[ \mathcal E(E) = \pi^{-1}\bigl[\operatorname{rad}\bigl(\cB(E)/\cA(E)\bigr)\bigr]\subseteq \pi^{-1}[\pi(\mathcal R)] = \mathcal R. \qedhere \] 
\end{proof}

\begin{corollary}\label{cor:approximable}
Every non-fixed maximal right ideal of $\cB(E)$ contains the ideals of approximable, compact and strictly singular operators on~$E$.
\end{corollary}

\begin{proof}
If $E$ is finite-dimensional, then $\cF(E)=\cB(E)$, so
Theorem~\ref{thm:A} shows that every maximal right ideal is fixed and
the assertion is vacuous.  Suppose therefore that $E$ is
infinite-dimensional.  The result follows from
Corollary~\ref{cor:strictly-singular} since
\[
  \cA(E)\subseteq\cK(E)\subseteq\cS(E)\subseteq\mathcal E(E). \qedhere 
\]
\end{proof}

\section{Finitely generated right ideals as lifting ideals}

\begin{proof}[Proof of Theorem~\ref{thm:B}]
\ref{thm:B:i}.  First, we prove that $\Lift(T)$ is a  finitely generated right ideal for every $n\in\N$ and $T\in\cB(E^n,E)$ by showing that
\begin{equation}\label{eq:Lift-generators}
  \Lift(T)=\sum_{j=1}^n(T\iota_j)\cB(E).
\end{equation}
This is easy to check: On the one hand, the identity
$\sum_{j=1}^n\iota_j\pi_j=\id_{E^n}$ implies that 
\[   TU=\sum_{j=1}^n(T\iota_j)(\pi_jU)\in\sum_{j=1}^n(T\iota_j)\cB(E) \]
for every $U\in\cB(E,E^n)$, 
and on the other, given $U_1,\ldots,U_n\in\cB(E)$, we see that the operator $U=\sum_{j=1}^n\iota_jU_j\in\cB(E,E^n)$ satisfies
\[ \sum_{j=1}^n(T\iota_j)U_j=TU\in\Lift(T). \]

Second, we verify that every finitely generated right ideal is a lifting ideal: Given $n\in\N$ and $T_1,\ldots,T_n\in\cB(E)$, we define 
\[ T=\sum_{j=1}^nT_j\pi_j\in\cB(E^n,E). \] 
Then $T\iota_j = T_j$ for each $1\leqslant j\leqslant n$, so~\eqref{eq:Lift-generators} implies that 
 \begin{equation*}
  \Lift(T) = \sum_{j=1}^nT_j\cB(E).
\end{equation*}

\ref{thm:B:ii}.  Suppose first that $T$ is surjective.  Given
$y\in E$ and $\lambda\in E^*$, choose $x\in E^n$ with $Tx=y$.  Then
$x\otimes\lambda\in\cF(E,E^n)$ and
\(
  y\otimes\lambda=T(x\otimes\lambda)\in\Lift(T).
\)
Thus $\cF(E)\subseteq\Lift(T)$.

Conversely, suppose that $\cF(E)\subseteq\Lift(T)$.  There is nothing
to prove if $E=\{0\}$, so choose $\lambda\in E^*$ and $x\in E$ with
$\lambda(x)=1$.  Given $y\in E$, choose
$U\in\cB(E,E^n)$ such that
\(
  y\otimes\lambda=TU.
\)
Evaluation at $x$ gives $y=TUx\in\ran T$, so $T$ is surjective.

\ref{thm:B:iii}.  Since $\Lift(T)$ is a right ideal,
$\Lift(T)=\cB(E)$ if and only if $\id_E\in\Lift(T)$, which is
precisely the assertion that $T$ is right invertible.

\ref{thm:B:iv}.  Since $T$ is not right invertible, $\Lift(T)$ is
proper by~\ref{thm:B:iii}.  For $S\in\cB(E)$, the right ideal
generated by $\Lift(T)$ and $S$ is
\[
  \Lift(T)+S\cB(E)
  =\{TU+SV:U\in\cB(E,E^n),\ V\in\cB(E)\}.
\]
It is equal to $\cB(E)$ if and only if there are
$U\in\cB(E,E^n)$ and $V\in\cB(E)$ such that
\(
  TU+SV=\id_E.
\)
This says exactly that the column operator
$x\mapsto(Ux,Vx)$ is a right inverse of $[T\ S]$.  The conclusion now
follows from the usual maximality criterion for proper right ideals.
\end{proof}
\begin{lemma}[Rows belonging to an operator ideal]
\label{lem:row-operator-ideal}
Let $\mathcal I$ be an operator ideal, let $E$ be a Banach space, and let
$T\in\cB(E^n,E)$.  The following conditions are equivalent:
\begin{enumerate}[label={\rm(\roman*)}]
\item $T\in\mathcal I(E^n,E)$;
\item $T\iota_j\in\mathcal I(E)$ for every $j=1,\ldots,n$;
\item $\Lift(T)\subseteq\mathcal I(E)$.
\end{enumerate}
Consequently, if $\Lift(T)=\mathcal I(E)$, then
$T\in\mathcal I(E^n,E)$.
\end{lemma}

\begin{proof}
If $T\in\mathcal I(E^n,E)$, then $TU\in\mathcal I(E)$ for every
$U\in\cB(E,E^n)$, so \rm(i) implies \rm(iii).  If \rm(iii) holds,
then $T\iota_j\in\Lift(T)\subseteq\mathcal I(E)$ for $1\leqslant j\leqslant n$, so \rm(iii) implies
\rm(ii).  Finally,
\[
  T=\sum_{j=1}^n(T\iota_j)\pi_j,
\]
so \rm(ii) implies \rm(i) because $\mathcal I$ is an operator ideal.
\end{proof}
\begin{proposition}\label{prop:individual-obstruction}
Let $E$ be a Banach space.  A subset $\mathcal R$ of $\cB(E)$ is a
non-fixed, finitely generated maximal right ideal if and only if
$\mathcal R=\Lift(T)$ for some $n\in\N$ and some surjective operator
$T\in\cB(E^n,E)$ such that $T$ is not right invertible and the row
operator
\[
  [T\ S]\colon E^n\oplus E\longrightarrow E,
  \qquad [T\ S](x,y)=Tx+Sy,
\]
is right invertible for every
$S\in\cB(E)\setminus\Lift(T)$.
\end{proposition}

\begin{proof}
Suppose first that $\mathcal{R}=\Lift(T)$ for some operator~$T$ with the stated properties. Then Theorem~\ref{thm:B}\ref{thm:B:i} shows that~$\mathcal{R}$ is finitely generated.  It is maximal by Theorem~\ref{thm:B}\ref{thm:B:iv}, and it is non-fixed by
Theorem~\ref{thm:A} and Theorem~\ref{thm:B}\ref{thm:B:ii}.

Conversely, let $\mathcal R$ be a non-fixed, finitely generated maximal
right ideal.  By Theorem~\ref{thm:B}\ref{thm:B:i}, we may write
$\mathcal R=\Lift(T)$ for some $T\in\cB(E^n,E)$.  Theorem~\ref{thm:A}
and Theorem~\ref{thm:B}\ref{thm:B:ii} show that $T$ is surjective;
properness and Theorem~\ref{thm:B}\ref{thm:B:iii} show that it is not
right invertible; and Theorem~\ref{thm:B}\ref{thm:B:iv} gives the row
completion property.
\end{proof}

\begin{corollary}[Exact obstruction]\label{cor:exact-obstruction}
The following conditions are equivalent for a Banach space~$E$:
\begin{enumerate}[label={\rm(\roman*)}]
\item $\cB(E)$ contains a non-fixed, finitely generated maximal right
ideal;
\item for some $n\in\N$, there is a surjective operator
$T\in\cB(E^n,E)$ which is not right invertible and for which $[T\ S]$
is right invertible for every $S\in\cB(E)\setminus\Lift(T)$.
\end{enumerate}
When these conditions hold, $\Lift(T)$ is a non-fixed, finitely
generated maximal right ideal for every operator $T$ satisfying the
second condition, and $\mathcal E(E)\subseteq\Lift(T)$.
\end{corollary}

\begin{proof}
The equivalence follows from
Proposition~\ref{prop:individual-obstruction}.  When the equivalent
conditions hold, $E$ is infinite-dimensional, because every surjection
between finite-dimensional spaces has a bounded linear right inverse.
The final inclusion now follows from
Corollary~\ref{cor:strictly-singular}.
\end{proof}

There is one further analytic consequence of maximality which is sometimes useful, but it is weaker than right invertibility.

\begin{proposition}\label{prop:uniform-lifts}
Let $T\in\cB(E^n,E)$ for some Banach space~$E$ and some $n\in\N$, and suppose that the right ideal~$\Lift(T)$ is closed. Then there is a constant $C\geqslant1$ such that
\[
  \inf\bigl\{\lVert U\rVert : U\in\cB(E,E^n),\, TU=S\bigr\}\leqslant C\lVert S\rVert
\]
for every $S\in\Lift(T)$.
In particular this holds when $\Lift(T)$ is maximal.
\end{proposition}
\begin{proof}
The definition~\eqref{eq:Lift} implies that the map $U\mapsto TU$ is a surjection of~$\cB(E,E^n)$ onto~$\Lift(T)$. It is clearly bounded and linear, so the existence of a constant satisfying the asserted estimate follows from the open mapping theorem.  Lemma~\ref{lem:max-closed} implies the final statement.
\end{proof}

\begin{remark}\label{rem:no-linear-selection}
Surjectivity of an operator from a Banach space~$E$ onto a Banach space~$F$ gives a constant $C>0$ such that every $y\in F$ has a preimage $x\in E$ with $\lVert x\rVert\leqslant C\lVert y\rVert$ by the open mapping theorem. However, these pointwise choices need not come from a globally defined bounded \emph{linear} right inverse $F\to E$.  Proposition~\ref{prop:uniform-lifts} likewise gives uniform control only for operators which are already liftable.  This is the precise gap in any argument which attempts to deduce right invertibility of~$T$ merely from maximality or closedness of $\Lift(T)$.
\end{remark}

\begin{remark}[Finite versus single generation]
\label{rem:principal}
Suppose that $E^n\cong E$ for some $n\geqslant2$.  Then every right
ideal of $\cB(E)$ which is generated by at most $n$ elements is
principal.

Indeed, after adjoining zero generators if necessary,
Theorem~\ref{thm:B}\ref{thm:B:i} shows that such a~right ideal has the
form
\(
  \mathcal R=\Lift(T)
\)
for some $T\in\cB(E^n,E)$.  Let
\(
  V\colon E\longrightarrow E^n
\)
be an isomorphism.  Every operator $U\in\cB(E,E^n)$ can be written
uniquely as
\[
  U=VA,
  \qquad A=V^{-1}U\in\cB(E).
\]
Consequently,
\[
  \mathcal R
  =\{TU:U\in\cB(E,E^n)\}
  =\{TVA:A\in\cB(E)\}
  =(TV)\cB(E).
\]
Thus $\mathcal R$ is generated by the single operator $TV$.

If, in addition, $\mathcal R$ is a non-fixed maximal right ideal, then
Theorem~\ref{thm:A} and
Theorem~\ref{thm:B}\ref{thm:B:ii} imply that $T$, and hence $TV$, is
surjective.  Since $\mathcal R$ is proper,
Theorem~\ref{thm:B}\ref{thm:B:iii} implies that $TV$ is not right
invertible.  Moreover, Theorem~\ref{thm:B}\ref{thm:B:iv} shows that
$\mathcal R=(TV)\cB(E)$ is maximal precisely when
\[
  [TV\ S]\colon E\oplus E\longrightarrow E
\]
is right invertible for every $S\in\cB(E)\setminus (TV)\cB(E)$.

In particular, if $E\cong E^2$, then $E\cong E^m$ for every
$m\in\N$, by induction.  Hence every finitely generated
right ideal of $\cB(E)$ is principal.  This observation does not,
however, imply that a surjective generator of a proper right ideal is
right invertible.
\end{remark}

Dales \emph{et al.}\ observed that the ideal $\mathcal W(E)$ of weakly
compact operators on a Banach space~$E$ may be proper and finitely
generated as a two-sided ideal, whereas it is finitely generated as a
left ideal only when $E$ is reflexive; see
\cite[Example~2.5 and Corollary~4.8]{DKKKL}.  The corresponding
right-sided statement takes the following form.

\begin{proposition}
\label{R:weaklycompactops}
Let $E$ be a Banach space.  The following conditions are equivalent:
\begin{enumerate}[label={\rm(\roman*)}]
\item\label{weaklycompact:right-fg}
the ideal $\mathcal W(E)$ is finitely generated as a
right ideal of $\cB(E)$;
\item\label{weaklycompact:right-principal}
the ideal $\mathcal W(E)$ is principal as a right ideal of $\cB(E)$;
\item\label{weaklycompact:all}
$\mathcal W(E)=\cB(E)$;
\item\label{weaklycompact:reflexive}
$E$ is reflexive.
\end{enumerate}
\end{proposition}

\begin{proof}
The equivalence of
\ref{weaklycompact:all} and \ref{weaklycompact:reflexive} is standard.
Indeed, if $E$ is reflexive, then every bounded operator on~$E$ is
weakly compact.  Conversely, if
\(
  \mathcal W(E)=\cB(E),
\)
then $\id_E$ is weakly compact, and hence the closed unit ball of~$E$
is weakly compact.

Condition~\ref{weaklycompact:all} plainly implies
\ref{weaklycompact:right-principal}, because in this case
\(
  \mathcal W(E)=\cB(E)=\id_E\cB(E),
\)
and \ref{weaklycompact:right-principal} implies
\ref{weaklycompact:right-fg}.

It remains to prove that \ref{weaklycompact:right-fg} implies
\ref{weaklycompact:reflexive}.  Suppose that $\mathcal W(E)$ is
finitely generated as a right ideal.  By
Theorem~\ref{thm:B}\ref{thm:B:i}, there are $n\in\N$ and
$T\in\cB(E^n,E)$ such that
\(
  \mathcal W(E)=\Lift(T).
\)
Since
\(
  \cF(E)\subseteq\mathcal W(E)=\Lift(T),
\)
Theorem~\ref{thm:B}\ref{thm:B:ii} shows that $T$ is surjective.
Moreover, Lemma~\ref{lem:row-operator-ideal}, applied to the operator
ideal of weakly compact operators, gives
\(
  T\in\mathcal W(E^n,E).
\)

The existence of a weakly compact surjection onto~$E$ forces~$E$ to be reflexive. 
This is well known, but for completeness, we include a short argument. Let~$B_E$ denote the closed unit ball of~$E$; it is norm closed and convex, and hence weakly closed. By the open mapping theorem, there is a constant $C>0$
such that
\(
  B_E\subseteq T(CB_{E^n}).
\)
It follows that~$B_E$ is a weakly closed subset of the weak closure of~$T(CB_{E^n})$, 
which is weakly compact because the operator~$T$ is weakly compact.  Hence $B_E$ is weakly compact, and therefore $E$ is reflexive.
\end{proof}

\section{The splitting criterion and its scope}

\begin{proof}[Proof of Theorem~\ref{thm:C}]
Suppose that $\cB(E)$ contains a non-fixed,
finitely generated maximal right ideal.  By
Corollary~\ref{cor:exact-obstruction}, it has the form
\(
  \Lift(T)
\)
for some $n\in\N$ and some surjective operator
\(
  T\in\cB(E^n,E)
\)
which is not right invertible.
This shows immediately that condition~\ref{thm:C:i} cannot be satisfied. 

Now assume, towards a contradiction, that condition~\ref{thm:C:ii} holds. Since $\Lift(T)$ is a maximal right ideal, it is closed by Lemma~\ref{lem:max-closed}.  Moreover, since it is non-fixed,
Theorem~\ref{thm:A} gives
\(
  \cF(E)\subseteq\Lift(T).
\)
Proposition~\ref{prop:uniform-lifts} therefore yields a constant
$C>0$ such that, for every non-zero $F\in\cF(E)$, there is an
operator $U_F\in\cB(E,E^n)$ satisfying
\begin{equation}\label{eq:uniform-finite-rank-lifts}
  TU_F=F
  \quad\text{and}\quad
  \lVert U_F\rVert<(C+1)\lVert F\rVert.
\end{equation}
For $F=0$, we take $U_0=0$.

Since $E$ has the bounded approximation property, there is a bounded
net $(F_\alpha)$ in $\cF(E)$ such that \(F_\alpha x\rightarrow x\) for every $x\in E$. It follows from~\eqref{eq:uniform-finite-rank-lifts} that the
corresponding net $(U_{F_\alpha})$ is bounded in $\cB(E,E^n)$.

By the standard duality for the projective tensor product
\cite[Proposition~A.3.70]{Dales}, we have an isometric identification
\[
  \left(E\widehat{\otimes}_{\pi}(E^n)^*\right)^*
  \cong
  \cB(E,(E^n)^{**}),
\]
where the duality bracket is given on elementary tensors by
\[
  \langle x\otimes\lambda,U\rangle
  =
  \lambda(Ux)
\]
for $x\in E$, $\lambda\in(E^n)^*$ and
$U\in\cB(E,(E^n)^{**})$.  Since $E$ is reflexive, so is $E^n$, and
hence this identifies $\cB(E,E^n)$ with the dual space
\[
  \bigl(E\widehat{\otimes}_{\pi}(E^n)^*\bigr)^*.
\]
The Banach--Alaoglu theorem now shows that, after passing to a subnet
and relabelling, we may suppose that $(U_{F_\alpha})$ converges in the
weak-star topology to an operator
\(
  U\in\cB(E,E^n).
\)

Let $x\in E$ and $\mu\in E^*$.  Since
\(
  \mu T\in(E^n)^*,
\)
we obtain
\begin{align*}
  \mu(TUx)
  &=
  \bigl\langle x\otimes (\mu T),U\bigr\rangle \\
  &=
  \lim_\alpha
  \bigl\langle x\otimes (\mu T),U_{F_\alpha}\bigr\rangle \\
  &=
  \lim_\alpha \mu T(U_{F_\alpha}x) \\
  &=
  \lim_\alpha \mu(F_\alpha x)
  =
  \mu(x).
\end{align*}
Since $x\in E$ and $\mu\in E^*$ were arbitrary, it follows that
\(
  TU=\id_E.
\)
Thus $T$ is right invertible, contrary to its choice.  Hence every
finitely generated maximal right ideal of $\cB(E)$ is
fixed.

When $E$ is infinite-dimensional, the final assertion follows from
Corollary~\ref{cor:fixed-implies-DZ}.
\end{proof}

\begin{proposition}\label{prop:examplesCi}
Let $E$ be isomorphic to either a Hilbert space or to~$\ell_1(\Gamma)$ for some index set~$\Gamma$. Then condition~\ref{thm:C:i} in Theorem~\ref{thm:C} is satisfied; that is, 
for every $n\in\N$, every surjective operator $T\in\cB(E^n,E)$ is right invertible.

Consequently, every finitely generated maximal right ideal of $\cB(E)$ is fixed, and~$\cB(E)$ contains a maximal right ideal which is not finitely generated, provided that~$E$ is infinite-dimensional.
\end{proposition}
    
\begin{proof}
Since condition~\ref{thm:C:i} is invariant under isomorphism, we may
suppose that $E$ is either a Hilbert space or $E=\ell_1(\Gamma)$ for
some set~$\Gamma$.

In the first case, equip $E^n$ with the Hilbertian $\ell_2$-sum norm,
which is equivalent to the product norm fixed in
Section~\ref{S:prelim}.  Then
$E^n=\ker T\oplus(\ker T)^\perp$.  The restriction of~$T$ to
$(\ker T)^\perp$ is a bounded bijection onto~$E$, so its inverse,
followed by the inclusion into~$E^n$, is a bounded right inverse for
$T$.

In the second case, let $(e_\gamma)_{\gamma\in\Gamma}$ denote the
canonical unit vector basis for~$\ell_1(\Gamma)$.  By the open mapping
theorem, there is $C>0$ such that, for every $\gamma\in\Gamma$, we may
choose $x_\gamma\in E^n$ with
$\lVert x_\gamma\rVert\leqslant C$ and $Tx_\gamma=e_\gamma$.  The
formula
\[
  U(a)=\sum_{\gamma\in\Gamma}a_\gamma x_\gamma
  \qquad(a=(a_\gamma)\in\ell_1(\Gamma))
\]
defines an operator $U\colon\ell_1(\Gamma)\to E^n$ with
$\lVert U\rVert\leqslant C$, and $TU=I_E$.

The final assertions follow from Theorem~\ref{thm:C}.
\end{proof}

The hypotheses of Theorem~\ref{thm:C} are sufficient, but condition \rm(i) is not stable under even very natural direct sums. To see this, we shall use the following standard fact; see, \emph{e.g.}, \cite[Theorem 2.3.1]{AlbiacKalton}.
\begin{lemma}\label{lem:l1-universal}
Every separable Banach space is a quotient of $\ell_1$.
\end{lemma}

\begin{example}[A non-right-invertible surjection]\label{ex:mixed}
Let $H$ be an infinite-dimensional separable Hilbert space, put
$L=\ell_1$ and $E=L\oplus_\infty H$.  By
Lemma~\ref{lem:l1-universal}, there is a surjection $q\colon L\to H$.
It is not right invertible, because a bounded linear right inverse would
embed $H$ into $\ell_1$, whereas subspaces of $\ell_1$ have the Schur
property.

Define $Q\colon E^2\to E$ by
\[
  Q\bigl((x_1,h_1),(x_2,h_2)\bigr)=(x_1,qx_2).
\]
Then $Q$ is surjective but not right invertible.  Indeed, if
$V\in\cB(E,E^2)$ were a right inverse, then composing $V$ with the
canonical embedding $\iota_H\colon H\to E$ and the projection~$P$ of $E^2$ onto the
$L$-coordinate of its second copy would give a right inverse for~$q$. This follows because $qP= \pi_HQ$, where $\pi_H\colon E\to H$ is the canonical projection, and hence \[ qPV\iota_H=\pi_HQV\iota_H=\id_H. \]
\end{example}

Thus condition~\ref{thm:C:i} can fail even for a very simple separable
Banach space.  However, Example~\ref{ex:mixed} does not produce a non-fixed, finitely generated maximal right ideal because maximality of~$\Lift(Q)$ would still
require the completion condition in Corollary~\ref{cor:exact-obstruction}.  In fact, as 
Theorem~\ref{thm:mixed-fixed} will show, every finitely generated
maximal right ideal of $\cB(E)$ \emph{is} fixed.

\begin{lemma}\label{lem:calkin-no-fg-max}
Let $H$ be an infinite-dimensional Hilbert space, and let
$\mathcal Q(H)=\cB(H)/\cK(H)$ be its Calkin algebra.  Then
$\mathcal Q(H)$ has no finitely generated maximal right
ideals.
\end{lemma}

\begin{proof}
Let $\pi\colon\cB(H)\to\mathcal Q(H)$ denote the quotient map, and
suppose that $M$ is a finitely generated maximal right
ideal of $\mathcal Q(H)$.  By Lemma~\ref{lem:max-closed}, $M$ is
closed, so \cite[Lemma~2.1]{BlecherKania} gives a projection
$p\in\mathcal Q(H)$ such that $M=p\mathcal Q(H)$.  Put $e=1-p$.
Since $M$ is proper and
\[
  \mathcal Q(H)=p\mathcal Q(H)\oplus e\mathcal Q(H)
\]
as right modules, $e\mathcal Q(H)$ is a non-zero simple right module.
It follows that $e$ is a minimal projection.  Indeed, if $f$ is a
projection with $0\neq f\leqslant e$, then $f\mathcal Q(H)$ is a
non-zero submodule of $e\mathcal Q(H)$, and hence equals it.  Writing $e=fa$ for some $a\in\mathcal Q(H)$, we obtain
$f=fe=f^2a=fa=e$.

We now show that $\mathcal Q(H)$ has no non-zero minimal projections.
Barnes proved more generally that every idempotent in
$\cB(X)/\cK(X)$ lifts to an idempotent in $\cB(X)$ for an arbitrary
Banach space~$X$; see
\cite[Lemma~1 and the remark following it]{BarnesAlgebraic}.  In the
Hilbert-space setting this can also be seen directly.  Choose
$B\in\cB(H)$ with $\pi(B)=e$ and put
$A=(B+B^*)/2$.  Then $A=A^*$, $\pi(A)=e$ and
$A^2-A\in\cK(H)$, so the essential spectrum of~$A$ is contained in
$\{0,1\}$.  Since spectral points of a self-adjoint operator outside
its essential spectrum are isolated eigenvalues of finite
multiplicity, $\sigma(A)\cap(0,1)$ is discrete and therefore countable.
We may consequently choose $t\in(0,1)\setminus\sigma(A)$.

The restriction of the characteristic function of $(t,\infty)$ to
$\sigma(A)$ is continuous. Continuous functional calculus therefore
gives an orthogonal projection
\(
  P=\mathbf 1_{(t,\infty)}(A);
\)
naturality of continuous functional calculus under the quotient
map implies that
\[
\pi(P)=\mathbf 1_{(t,\infty)}(e)=e. 
\]
Since $e\neq0$, the range of~$P$ is infinite-dimensional.  Decompose it
orthogonally as $\ran P=H_1\oplus H_2$, where both summands are
infinite-dimensional, and let $P_1$ and $P_2$ be the corresponding
orthogonal projections, so $P=P_1+P_2$.  Neither $P_1$ nor $P_2$ is
compact.  Hence $\pi(P_1)$ is a non-zero proper subprojection of~$e$,
contradicting minimality.
\end{proof}

\begin{theorem}
\label{thm:mixed-fixed}
Let $\Gamma$ be a non-empty set, let $H$ be an infinite-dimensional
separable Hilbert space, and put
\[
  L=\ell_1(\Gamma)
  \quad\text{and}\quad
  E=L\oplus_\infty H.
\]
Then every finitely generated maximal right ideal
of~$\cB(E)$ is fixed.  Consequently, $\cB(E)$ contains a maximal
right ideal which is not finitely generated.
\end{theorem}

\begin{proof}
Suppose, towards a contradiction, that $\mathcal R$ is a 
non-fixed, finitely generated maximal right ideal
of~$\cB(E)$.  By Corollary~\ref{cor:exact-obstruction}, there are
$n\in\N$ and a surjection $T\in\cB(E^n,E)$ such that
\(
  \mathcal R=\Lift(T).
\)
Let
\(
  \iota_L\colon L\to E
\)
and
\(
  \pi_L\colon E\to L
\)
denote the canonical coordinate embedding and projection,
respectively, and write $(e_\gamma)_{\gamma\in\Gamma}$ for the unit vector basis of~$L$.  By the open mapping theorem, for each $\gamma\in\Gamma$, we may choose $z_\gamma\in E^n$ such that
\[
  Tz_\gamma=\iota_L e_\gamma
  \quad\text{and}\quad
  \sup_{\gamma\in\Gamma}\lVert z_\gamma\rVert<\infty.
\] 
Consequently, the formula
\[
  V(a)
  =\sum_{\gamma\in\Gamma}a_\gamma z_\gamma
  \qquad (a=(a_\gamma)\in L)
\]
defines an operator $V\in\cB(L,E^n)$ satisfying
\(
  TV=\iota_L.
\)
It follows that the coordinate projection
\[
  P=\iota_L\pi_L
   =
  \begin{pmatrix}
    I_L&0\\
    0&0
  \end{pmatrix}
\]
belongs to~$\mathcal R$, because
\(
  P=TV\pi_L\in\Lift(T).
\)

As above, we write operators on $E=L\oplus H$ in block-matrix form.  Every operator
from~$H$ to~$L$ is compact.  Indeed, its range is separable and is
therefore contained in $\ell_1(\Gamma_0)$ for some countable subset
$\Gamma_0$ of~$\Gamma$, after which Pitt's theorem~\cite{Pitt} applies.
Consequently,
\[
  \Phi\colon\cB(E)\longrightarrow\mathcal Q(H),
  \qquad
  \Phi
  \begin{pmatrix}
    A&B\\
    C&D
  \end{pmatrix}
  =D+\cK(H),
\]
is a continuous surjective unital algebra homomorphism.  To see
multiplicativity, observe that the lower right-hand corner of a product is
\(
  C_1B_2+D_1D_2,
\)
and $C_1B_2$ is compact.

We claim that $\ker\Phi\subseteq\mathcal R$.  Since $P\in\mathcal R$,
we have $P\cB(E)\subseteq\mathcal R$.  We have already seen that
every operator from $H$ to $L$ is compact, and hence inessential.  By
the left-right symmetry of inessentiality noted above, every operator
from $L$ to $H$ is also inessential.  Indeed, if
$C\in\cB(L,H)$ and $B_0\in\cB(H,L)$, then $B_0$ is compact, so
$I_L-B_0C$ is Fredholm.

Now take
\[
  S=
  \begin{pmatrix}
    A&B\\
    C&D
  \end{pmatrix}
  \in\ker\Phi.
\]
Then $D\in\cK(H)$, and
\[
  S
  =
  PS+
  \begin{pmatrix}
    0&0\\
    C&0
  \end{pmatrix}
  +
  \begin{pmatrix}
    0&0\\
    0&D
  \end{pmatrix}.
\]
The first summand belongs to $\mathcal R$.  The second belongs to
$\mathcal E(E)$ by the operator-ideal property of the inessential
operators, while the third is compact and hence also inessential.
Corollary~\ref{cor:strictly-singular} therefore shows that the last two
summands belong to $\mathcal R$ as well.  Hence $S\in\mathcal R$, proving
the claim.

Since $\ker\Phi\subseteq\mathcal R$, the quotient correspondence shows
that
\(
  M=\Phi(\mathcal R)
\)
is a maximal right ideal of~$\mathcal Q(H)$ and that
\(
  \mathcal R=\Phi^{-1}(M).
\)
Applying~$\Phi$ to a finite generating set for~$\mathcal R$ shows that
$M$ is finitely generated.  This contradicts
Lemma~\ref{lem:calkin-no-fg-max}.  Therefore every finitely generated maximal right ideal of~$\cB(E)$ is fixed.

The final assertion follows from
Corollary~\ref{cor:fixed-implies-DZ}.
\end{proof}

\section{Boolean families and many maximal right ideals}

\noindent We first isolate the key algebraic counting argument.  A \emph{Boolean family of idempotents indexed by $\mathcal{P}(\N)$} in a unital algebra $A$ is a family $(p_S)_{S\subseteq\N}$ of idempotents in~$A$  such that
\[
  p_\varnothing=0,\qquad p_\N=1_A,\qquad
  p_Sp_T=p_{S\cap T},\qquad
  p_{\N\setminus S}=1_A-p_S,
\]
and $p_S\neq0$ whenever $S\neq\varnothing$.

\begin{proposition}\label{prop:boolean}
Let~$A$ be a unital algebra which contains a Boolean family of idempotents indexed by~$\mathcal{P}(\N)$. Then~$A$ contains at least $2^{\mathfrak{c}}$ maximal right ideals.
\end{proposition}

\begin{proof}
For an ultrafilter $\mathcal{U}$ on $\N$, let $I_\mathcal{U}$ be the right ideal generated by the set
\(
  \{p_S:S\notin\mathcal{U}\}.
\)
To see that~$I_{\mathcal{U}}$ is proper, assume the contrary, and write 
\begin{equation}\label{prop:boolean:eq1}
 1_A=\sum_{j=1}^mp_{S_j}a_j   
\end{equation}
for some $m\in\N$, $S_1,\ldots,S_m\in\mathcal{P}(\N)\setminus\mathcal{U}$ and $a_1,\ldots,a_m\in A$. Then
\[
  T=\N\setminus\bigcup_{j=1}^mS_j\in\mathcal{U}.
\]
Hence $T\neq\varnothing$, so $p_T\neq0$; however, multiplying the identity~\eqref{prop:boolean:eq1} on the left by~$p_T$, we obtain $p_T=0$, a~contra\-dic\-tion.

Choose a maximal right ideal $M_\mathcal{U}$ of~$A$ containing~$I_\mathcal{U}$. We claim that the map $\mathcal{U}\mapsto M_\mathcal{U}$ is injective. Indeed, if~$\mathcal{U}$ and~$\mathcal{V}$ are distinct ultrafilters on~$\N$, choose $S\in\mathcal{U}\setminus\mathcal{V}$.  Then $p_{\N\setminus S}\in M_\mathcal{U}$ and $p_S\in M_\mathcal{V}$.  Were the two maximal ideals equal, they would contain
\[
  p_S+p_{\N\setminus S}=1_A,
\]
a contradiction.  Thus distinct ultrafilters give distinct maximal right ideals. Now the conclusion follows from the fact that there are $2^{\mathfrak{c}}$ ultrafilters on $\N$; see also the corresponding left-sided argument in \cite{DKKKL}.
\end{proof}

\begin{proof}[Proof of Theorem~\ref{thm:D}]
Let
\[
  E=\bigoplus_{n=1}^\infty E_n
\]
be an unconditional Schauder decomposition with each closed subspace $E_n\neq\{0\}$.  For $S\subseteq\N$, let $P_S$ be the coordinate projection onto $\bigoplus_{n\in S}E_n$.  Unconditionality ensures that~$P_S\in\cB(E)$, and $(P_S)_{S\subseteq\N}$ is a Boolean family of idempotents indexed by~$\mathcal{P}(\N)$. Hence $\cB(E)$ contains at least $2^{\mathfrak{c}}$ maximal right ideals by Proposition~\ref{prop:boolean}.

Since $E$ is separable, a bounded operator on $E$ is determined by its values on a countable dense subset.  Therefore
\(
  |\cB(E)|=\mathfrak{c}.
\)
It follows that~$\cB(E)$ contains~$2^{\mathfrak{c}}$ subsets, and therefore at most~$2^{\mathfrak{c}}$ maximal right ideals. Furthermore, there are at most $\mathfrak{c}$ finite subsets of~$\cB(E)$, and consequently at most $\mathfrak{c}$ finitely generated right ideals.  Removing these from a family of cardinality $2^{\mathfrak{c}}$ leaves $2^{\mathfrak{c}}$ non-finitely-generated maximal right ideals.
\end{proof}

\begin{corollary}\label{cor:complemented-unconditional}
The conclusion of Theorem~\ref{thm:D} holds if $E$ is separable and
contains a complemented subspace with an unconditional Schauder
decomposition into a countably infinite sequence of non-zero subspaces.
\end{corollary}

\begin{proof}
Write $E=Y\oplus Z$, where $Y=\bigoplus_{n\geqslant1}Y_n$ is an unconditional Schauder decomposition.  Then
\[
  E=(Z\oplus Y_1)\oplus\bigoplus_{n\geqslant2}Y_n
\]
is an unconditional Schauder decomposition of $E$ into non-zero subspaces. Therefore Theorem~\ref{thm:D} applies.
\end{proof}
\section{Reflexive spaces, Lebesgue spaces and Orlicz spaces}
\begin{theorem}\label{thm:reflexive}
Let $E$ be an infinite-dimensional reflexive Banach space.  Then
$\cB(E)$ contains a maximal right ideal which is not finitely generated.
\end{theorem}

\begin{proof}
The adjoint map is an isometric algebra isomorphism
\[
  \cB(E)\longrightarrow\cB(E^*)^{\op},
  \qquad T\longmapsto T^*,
\]
because $E$ is reflexive.  White proved that $\cB(E^*)$ has a maximal
left ideal which is not finitely generated
\cite[Corollary~2.2.7]{WhiteThesis}.  Proposition~\ref{prop:opposite}
therefore gives a non-finitely-generated maximal right ideal of
$\cB(E^*)^{\op}$; taking its preimage under the displayed isomorphism
gives the result.
\end{proof}

\begin{proposition}\label{prop:complemented-l1}
Let $E$ be a Banach space which contains a complemented copy of
$\ell_1$.  Then $\cB(E)$ contains a proper right ideal $\mathcal J$
with the following properties:
\begin{enumerate}[label={\rm(\roman*)}]
\item\label{prop:complemented-l1:approximable}
\(\cA(E)\subseteq\overline{\mathcal{J}};\) 
\item\label{prop:complemented-l1:rigidity}
$\cB(E)$ is the only closed, finitely generated right
ideal which contains $\mathcal J$.
\end{enumerate}

Consequently, no proper closed right ideal of~$\cB(E)$ containing $\mathcal J$
is finitely generated.  In particular,
$\overline{\mathcal J}$ is a proper closed right ideal which is not
finitely generated, and every maximal right ideal
containing $\mathcal J$ is non-fixed and not finitely
generated.
\end{proposition}

\begin{proof}
Choose a complemented subspace $F$ of $E$, an isomorphism
\(
  U\colon \ell_1\to F,
\)
and a bounded projection $P\colon E\to F$.  Put
\(
  Q=\id_E-P.
\)
Let $(e_n)_{n\in\N}$ denote the canonical basis of $\ell_1$, and let
$(e_n^*)_{n\in\N}$ denote the corresponding coordinate functionals.
For each $n\in\N$, define
\[
  R_n
  =
  U(e_n\otimes e_n^*)U^{-1}P
  \in\cB(E).
\]
A direct calculation shows that
\begin{equation}\label{prop:complemented-l1:eq1}
  R_mR_n=\delta_{m,n}R_n
  \quad\text{and}\quad
  R_mQ=0
  \qquad(m,n\in\N).
\end{equation}
Thus the operators $R_n$ are non-zero, pairwise orthogonal
projections.  Moreover, the canonical basis expansion in $\ell_1$
gives
\begin{equation}\label{eq:complemented-l1-expansion}
  Px=\sum_{n=1}^\infty R_nx
  \qquad(x\in E),
\end{equation}
where the series converges in norm.

Set
\(
  K=\lVert U\rVert\,\lVert U^{-1}\rVert\,\lVert P\rVert.
\)
Since $\lVert e_n\otimes e_n^*\rVert=1$, the definition of $R_n$
immediately gives
\[
  \sup_{n\in\N}\lVert R_n\rVert\leqslant K.
\]
We shall also use the stronger pointwise estimate
\begin{align}\label{eq:complemented-l1-summability}
  \sum_{n=1}^\infty\lVert R_nx\rVert
  &=
  \sum_{n=1}^\infty
  \left\|
    U\bigl(e_n^*(U^{-1}Px)e_n\bigr)
  \right\|\\
  &\leqslant
  \lVert U\rVert
  \sum_{n=1}^\infty
  \left|e_n^*(U^{-1}Px)\right| \notag\\
  &=
  \lVert U\rVert\,\lVert U^{-1}Px\rVert_{\ell_1}
  \leqslant K\lVert x\rVert
  \qquad(x\in E).\notag  
\end{align}

Consider the right ideal
\[
  \mathcal J
  =
  Q\cB(E)
  +
  \left\{
    \sum_{n\in F_0}R_nS_n:
    F_0\subseteq\N\text{ finite},\
    S_n\in\cB(E)\ (n\in F_0)
  \right\}.
\]
We first show that $\mathcal J$ is proper.  Otherwise, there would be a
finite subset $F_0$ of $\N$ and operators
$S_0,S_n\in\cB(E)$, for $n\in F_0$, such that
\[
  \id_E
  =
  QS_0+\sum_{n\in F_0}R_nS_n.
\]
Choose $m\in\N\setminus F_0$.  Composition on the left by $R_m$
gives
\[
  R_m
  =
  R_mQS_0+\sum_{n\in F_0}R_mR_nS_n
  =0
\]
by~\eqref{prop:complemented-l1:eq1}, a contradiction.

We next prove~\ref{prop:complemented-l1:approximable}.  Let
$x\in E$ and $\lambda\in E^*$.  We have
\(
  x\otimes\lambda
  =
  Q(x\otimes\lambda)+P(x\otimes\lambda).
\)
The first summand belongs to $Q\cB(E)\subseteq\mathcal J$.  By
\eqref{eq:complemented-l1-expansion},
\[
  \sum_{n=1}^N R_n(x\otimes\lambda)
  =
  \left(\sum_{n=1}^N R_nx\right)\otimes\lambda
  \underset{N\to\infty}{\longrightarrow}
  (Px)\otimes\lambda
  =
  P(x\otimes\lambda)
\]
in operator norm.  Each partial sum on the left-hand side belongs to $\mathcal J$, so
$P(x\otimes\lambda)\in\overline{\mathcal{J}}$. Hence $x\otimes\lambda\in\overline{\mathcal{J}}$, and therefore  
\(
  \cF(E)\subseteq\overline{\mathcal J}.
\)
Since $\overline{\mathcal J}$ is closed, it follows that
\(
  \cA(E)=\overline{\cF(E)}
  \subseteq\overline{\mathcal J}.
\)

We now establish~\ref{prop:complemented-l1:rigidity}.  Let
$\mathcal R$ be a closed,  finitely generated right ideal
such that
\(
  \mathcal J\subseteq\mathcal R.
\)
By Theorem~\ref{thm:B}\ref{thm:B:i}, 
\(
  \mathcal R=\Lift(T)
\) for some $r\in\N$ and 
\(
  T\in\cB(E^r,E).
\)
Since $\mathcal R$ is closed,
Proposition~\ref{prop:uniform-lifts} provides a constant $C\geqslant1$
such that
\[
  \inf\left\{
    \lVert V\rVert:
    V\in\cB(E,E^r),\ TV=S
  \right\}
  \leqslant C\lVert S\rVert
  \qquad(S\in\mathcal R).
\]
For every $n\in\N$, $R_n$ belongs to $\mathcal J$ and therefore to $\mathcal R$, so we can choose 
\mbox{\(
  V_n\in\cB(E,E^r)
\)}
such that
\(
  TV_n=R_n
\)
and
\(
  \lVert V_n\rVert
  \leqslant C\lVert R_n\rVert+1
  \leqslant CK+1.
\)
Put
\(
  D=CK+1.
\)
By~\eqref{eq:complemented-l1-summability}, we have 
\[
  \sum_{n=1}^\infty\lVert V_nR_nx\rVert
  \leqslant
  D\sum_{n=1}^\infty\lVert R_nx\rVert
  \leqslant DK\lVert x\rVert
\]
for every $x\in E$, so we can define an operator \(
  V\in\cB(E,E^r)
\)
by 
\[
  Vx=\sum_{n=1}^\infty V_nR_nx
  \qquad(x\in E).
\]

Using only idempotence of $R_n$ and
\eqref{eq:complemented-l1-expansion}, we obtain
\begin{align*}
  TVx
  &=
  \sum_{n=1}^\infty TV_nR_nx
   =
  \sum_{n=1}^\infty R_n^2x =  \sum_{n=1}^\infty R_nx
   =
  Px
  \qquad(x\in E).
\end{align*}
Hence
\(
  P=TV\in\Lift(T)=\mathcal R.
\)
Since $Q\in\mathcal J\subseteq\mathcal R$, we conclude that
\(
  \id_E=P+Q\in\mathcal R.
\)
Therefore
\(
  \mathcal R=\cB(E),
\)
which completes the proof of~\ref{prop:complemented-l1:rigidity}.

Since $\mathcal J$ is proper, it is contained in a maximal
right ideal $\mathcal M$ of $\cB(E)$.  By
Lemma~\ref{lem:max-closed}, $\mathcal M$ is closed, and therefore
\(
  \overline{\mathcal J}\subseteq\mathcal M.
\)
In particular, $\overline{\mathcal J}$ is proper.  It cannot be
 finitely generated by~\ref{prop:complemented-l1:rigidity}.  The same argument applies to
every maximal right ideal containing $\mathcal J$.

Finally,
\(
  \cF(E)\subseteq\overline{\mathcal J}\subseteq\mathcal M,
\)
so every maximal right ideal containing~$\mathcal{J}$ is non-fixed by
Theorem~\ref{thm:A}.
\end{proof}

In line with our general focus on complex scalars, we follow the non-standard convention that a ``Banach lattice'' means a complex Banach
lattice.  For such a~lattice~$E$, we write $E_{\mathbb R}$ for its
canonical real part, so that every element of~$E$ has a unique
representation
\(
  x+\mathrm{i}y,
\)
where \(
x,y\in E_{\mathbb R}.
\)
It is an elementary fact that~$E$ is reflexive if and only if~$E_{\mathbb R}$ is reflexive.

All order-theoretic terminology concerning~$E$, including positivity and order continuity, refers to the corresponding notions for the real Banach lattice~$E_{\mathbb R}$. In particular, the positive cone is $E_+ = (E_{\mathbb{R}})_+$, and a \emph{KB-space} is a (complex) Banach lattice~$E$ for which every norm-bounded, increasing sequence in $E_+$ converges in norm.

\begin{theorem}\label{thm:KB-lattice}
Let $E$ be an infinite-dimensional KB-space.  Then $\cB(E)$
contains a~maximal right ideal which is not  finitely
generated.

If $E$ is non-reflexive, then it contains a complemented copy of~$\ell_1$, so the stronger conclusion of
Proposition~\ref{prop:complemented-l1} holds.
\end{theorem}

\begin{proof}
If~$E$ is reflexive,
the conclusion follows from Theorem~\ref{thm:reflexive}.

Otherwise~$E_{\mathbb R}$ is a
non-reflexive real KB-space.  Every KB-space is weakly sequentially
complete, and therefore contains no subspace isomorphic to~$c_0$.
By Lozanovski\u{\i}'s reflexivity theorem, a non-reflexive real Banach
lattice contains a~sublattice isomorphic to either $c_0$ or~$\ell_1$; see \cite[Chapter~2]{MeyerNieberg}.  It follows that
$E_{\mathbb R}$ contains a~copy of real~$\ell_1$.

The norm of a KB-space is order continuous.  Hence
\cite[Corollary~1]{Wojtowicz} shows that $E_{\mathbb R}$ contains a~complemented subspace~$Y$ isomorphic to real~$\ell_1$.  Let
\(
  P_{\mathbb R}\colon E_{\mathbb R}\to Y
\)
be a bounded real-linear projection.  Its complexification
\[
  P_{\mathbb C}\colon E\longrightarrow E,
  \qquad
  P_{\mathbb C}(x+\mathrm{i}y)
  =
  P_{\mathbb R}x+\mathrm{i}P_{\mathbb R}y
  \quad(x,y\in E_{\mathbb R}),
\]
is a bounded complex-linear projection onto
\(
  Y_{\mathbb C}=Y+\mathrm{i}Y.
\)
Complexifying an isomorphism from real $\ell_1$ onto~$Y$ shows that
$Y_{\mathbb C}$ is isomorphic to complex~$\ell_1$.  Thus $E$ contains
a~complemented copy of complex~$\ell_1$, and
Proposition~\ref{prop:complemented-l1} applies.
\end{proof}

\begin{corollary}\label{cor:Lebesgue-spaces}
The right-sided Dales--\.{Z}elazko conjecture holds for~$\cB(L_p(\mu))$ for every $1\leqslant p<\infty$ and every measure space $(\Omega,\Sigma,\mu)$. 
\end{corollary}

\begin{proof}
The case where $L_p(\mu)$ is finite-dimensional is clear. For $1<p<\infty$, the space
$L_p(\mu)$ is reflexive, so the conclusion follows from
Theorem~\ref{thm:reflexive}.  Suppose
therefore that $p=1$ and $L_1(\mu)$ is infinite-dimensional.  Let
$(f_n)_{n\in\N}$ be a norm-bounded, increasing sequence in
$L_1(\mu)_+$, and put $f=\sup_n f_n$ pointwise.  The monotone
convergence theorem gives
\[ 
f\in L_1(\mu)  \quad\text{and}\quad
  \lVert f-f_n\rVert_1
  =\lVert f\rVert_1-\lVert f_n\rVert_1
  \longrightarrow0.
\]
Thus $L_1(\mu)$ is a KB-space, and
Theorem~\ref{thm:KB-lattice} applies. 
\end{proof}

Since terminology concerning Banach function spaces is not entirely
uniform, we fix the convention used here.  Let
$(\Omega,\Sigma,\mu)$ be a measure space, and let $L_0(\mu)$ denote
the space of equivalence classes of complex-valued measurable functions
on~$\Omega$.  A \emph{Banach function space} over
$(\Omega,\Sigma,\mu)$ is a vector subspace~$E$ of~$L_0(\mu)$ equipped with a complete norm~$\lVert\cdot\rVert_E$ satisfying the following ideal property: if $f\in E$ and $g\in L_0(\mu)$ satisfy
\[
  |g|\leqslant |f|
  \quad\text{almost everywhere},
\]
then $g\in E$ and
\(
  \lVert g\rVert_E\leqslant\lVert f\rVert_E.
\)
With the almost-everywhere pointwise order, such a space is a Banach
lattice.
We say that a Banach function space~$E$ has the \emph{Fatou property} if, whenever
$(f_n)_{n\in\N}$ is a norm-bounded, increasing sequence in $E_+$ which converges pointwise to a function~$f$ almost every\-where,  
we have $f\in E$ and
\[
  \lVert f\rVert_E
  =
  \sup_{n\in\N}\lVert f_n\rVert_E.
\]

\begin{corollary}
\label{cor:fatou-order-continuous}
Let $E$ be an
infinite-dimensional Banach function space.  Suppose that~$E$ has the Fatou property and
that its norm is order continuous.  Then~$E$ is a KB-space. 

Consequently~$\cB(E)$ contains a maximal
right ideal which is not  finitely generated.
\end{corollary}

\begin{proof}
Let $(f_n)_{n\in\N}$ be a norm-bounded, increasing sequence in $E_+$,
put
\[
  C=\sup_{n\in\N}\lVert f_n\rVert_E,
  \qquad
  f=\sup_{n\in\N}f_n
\]
pointwise, where at first $f$ is allowed to take the value $+\infty$.
For each $k\in\N$, the sequence $(f_n\wedge k)_{n\in\N}$ is
norm-bounded in $E_+$ and increases pointwise to the finite-valued
function $f\wedge k$.  The Fatou property therefore gives
\[
  f\wedge k\in E
  \quad\text{and}\quad
  \lVert f\wedge k\rVert_E
  =\sup_{n\in\N}\lVert f_n\wedge k\rVert_E
  \leqslant C.
\]
Let $A=\{f=+\infty\}$.  Since $\mathbf 1_A\leqslant f\wedge1$, the
ideal property gives $\mathbf 1_A\in E$.  Moreover,
$k\mathbf 1_A\leqslant f\wedge k$, and consequently
\[
  k\lVert\mathbf 1_A\rVert_E\leqslant C
  \qquad(k\in\N).
\]
Hence $\mathbf 1_A=0$ in $E$, so $A$ is a null set and $f$ is finite almost
everywhere.  We may now apply the Fatou property to $(f_n)$ itself to
obtain
\[
  f\in E
  \quad\text{and}\quad
  \lVert f\rVert_E
  =
  \sup_{n\in\N}\lVert f_n\rVert_E.
\]
Since
\[
  0\leqslant f-f_n\downarrow0
  \quad\text{almost everywhere},
\]
order continuity of the norm implies that
\(
  \lVert f-f_n\rVert_E\rightarrow0.
\)
Thus every norm-bounded, increasing sequence in $E_+$ converges in norm,
so $E$ is a KB-space.  

The final sentence follows from
Theorem~\ref{thm:KB-lattice}.
\end{proof}

\begin{corollary}\label{cor:Orlicz-order-continuous}
Let $(\Omega,\Sigma,\mu)$ be a measure space, let $\Phi$ be a Young
function, and suppose that the full Orlicz space $L^\Phi(\mu)$, equipped
with the Luxemburg norm, is infinite-dimensional and has
order-continuous norm.  Then
$\cB(L^\Phi(\mu))$ contains a maximal right ideal which is not
 finitely generated.

In particular, if $\mu$ is finite and non-atomic and $\Phi$ is
finite-valued, the conclusion holds whenever $\Phi$ satisfies the
$\Delta_2$-condition at infinity; equivalently, there are constants
$K>0$ and $t_0\geqslant0$ such that
\[
  \Phi(2t)\leqslant K\Phi(t)
  \qquad(t\geqslant t_0).
\]
\end{corollary}

\begin{proof}
Every full Orlicz space with the Luxemburg norm has the Fatou property,
so the first assertion follows from
Corollary~\ref{cor:fatou-order-continuous}.  On a finite non-atomic
measure space, for a~finite-valued Young function, order continuity of
the Luxemburg norm is equivalent to the $\Delta_2$-condition at
infinity; see \cite{Musielak}.
\end{proof}

We conclude this section with the observation that the KB hypothesis in Theorem~\ref{thm:KB-lattice} may be replaced with separability as long as order continuity is retained.

\begin{proposition}\label{prop:separable-order-continuous-lattice}
Let $E$ be an infinite-dimensional separable Banach lattice with
order-continuous norm.  Then $\cB(E)$ contains a maximal right ideal
which is not  finitely generated.
\end{proposition}

\begin{proof}
By Theorem~\ref{thm:reflexive}, it suffices to consider the non-reflexive case.  Then the real part~$E_{\mathbb R}$ is a non-reflexive separable real Banach lattice with
order-continuous norm.  By Lozanovski\u{\i}'s theorem,
$E_{\mathbb R}$ contains a copy of real~$\ell_1$ or real~$c_0$.

In the former case, we can argue as in the proof of Theorem~\ref{thm:KB-lattice}: 
\cite[Corollary~1]{Wojtowicz} gives a complemented copy of real $\ell_1$ in $E_{\mathbb R}$. Complexifying
the corresponding projection, we obtain a complemented copy of complex~$\ell_1$ in~$E$, so Proposition~\ref{prop:complemented-l1} applies.

In the latter case, Sobczyk's theorem~\cite{Sobczyk} gives a complemented copy of real
$c_0$ in $E_{\mathbb R}$.  Its complexification is a complemented copy
of complex $c_0$ in~$E$, and Corollary~\ref{cor:complemented-unconditional} applies.
\end{proof}

\section{Idempotents and a false shortcut}
\noindent
For $\cB(E)$, idempotent-generated maximal right ideals are completely rigid.

\begin{proposition}\label{prop:idempotent-BE}
Let $E$ be a Banach space and let $P\in\cB(E)$ be an idempotent.  Then
$P\cB(E)$ is a maximal right ideal if and only if
$\dim\ker P=1$.  In that case $P\cB(E)$ is fixed.
\end{proposition}

\begin{proof}
Put $Q=I_E-P$.  Since $P$ is idempotent, $\ran Q=\ker P$, and $\cB(E)$ decomposes as 
a direct sum of right $\cB(E)$-modules
\[
  \cB(E)=P\cB(E)\oplus Q\cB(E).
\]
Consequently, $P\cB(E)$ is maximal if and only if $Q\cB(E)$ is a
non-zero simple right module.

If $\ker P=\{0\}$, then $P=I_E$, so $P\cB(E)$ is not proper.  

Suppose that $\dim\ker P=1$.  Write $Q=y\otimes\lambda$, where
$\lambda(y)=1$.  Every non-zero element of $Q\cB(E)$ has the form
$y\otimes\mu$ with $\mu\neq0$.  Choosing $x\in E$ with $\mu(x)=1$
gives 
\[ (y\otimes\mu)(x\otimes\lambda)=Q. \]  Thus every non-zero
element generates $Q\cB(E)$, so this module is simple. Furthermore, in this case,
$P=I_E-y\otimes\lambda$ and
$P\cB(E)=\mathcal M^r_\lambda$ by Proposition~\ref{prop:fixed}; this proves the final statement.

Finally, suppose that $\dim\ker P>1$.  Choose
$0\neq z\in\ker P$ and $\nu\in E^*$ with $\nu(z)=1$, and put
$F=z\otimes\nu$.  Then $QF=F$, so $F\cB(E)$ is a non-zero submodule of
$Q\cB(E)$.  Every operator in $F\cB(E)$ has range contained in
$\mathbb Cz$, whereas $Q$ does not.  Hence $F\cB(E)\subsetneq Q\cB(E)$, and therefore
$Q\cB(E)$ is not simple.
\end{proof}

There is a useful general criterion, but its hypothesis cannot simply be assumed.

\begin{theorem}\label{thm:idempotent-criterion}
Let $A$ be a unital Banach algebra, and suppose that every finitely
generated maximal right ideal of~$A$ is generated by an idempotent.  If
every maximal right ideal of~$A$ is finitely generated, then $A$ is
finite-dimensional.  Equivalently, every infinite-dimensional Banach
algebra satisfying the idempotent hypothesis has a maximal right ideal
which is not finitely generated.
\end{theorem}

\begin{proof}
Under the hypotheses, every maximal right ideal has the form $pA$ for
an idempotent~$p$.  Since
$A=pA\oplus(1-p)A$ as right modules, maximality of $pA$ implies that
$(1-p)A$ is a non-zero simple right module.

Let $\Sigma$ be the sum of all simple right ideals of~$A$.  For every
maximal right ideal $M=pA$, the complementary simple right ideal
$(1-p)A$ is not contained in~$M$, so~$M$ does not contain~$\Sigma$. Hence $\Sigma=A$ because every proper right ideal is contained in a maximal one. 

Since $1\in\Sigma$, there are simple right ideals
$R_1,\ldots,R_m$ such that $1\in R_1+\cdots+R_m$, and therefore
$A=R_1+\cdots+R_m$.  Every simple module is Noetherian, and a finite
sum of Noetherian submodules is Noetherian by
\cite[Corollary~X.1.3]{Lang}.  Hence the right module $A_A$ is
Noetherian.  Equivalently, $A^{\op}$ is a left-Noetherian Banach
algebra.  The theorem of Sinclair and Tullo~\cite{SinclairTullo}
therefore implies that $A^{\op}$, and hence $A$, is finite-dimensional.
\end{proof}

\begin{example}[Upper-triangular matrices]\label{ex:triangular}
Let
\[
  A=T_2(\C)
  =\left\{\begin{pmatrix}a&b\\0&c\end{pmatrix}:a,b,c\in\C\right\}
\]
and
\[
  M=\left\{\begin{pmatrix}0&x\\0&y\end{pmatrix}:x,y\in\C\right\}
  =e_{12}A+e_{22}A.
\]
Then $M$ is a maximal right ideal generated by two elements, but it is
not singly generated and therefore is not generated by an idempotent.

Indeed, $A/M\cong\C$, so $M$ is maximal.  If
\[
  0\neq m=\begin{pmatrix}0&x\\0&y\end{pmatrix}\in M,
\]
then, for every $a,b,c\in\C$,
\[
  m\begin{pmatrix}a&b\\0&c\end{pmatrix}
  =\begin{pmatrix}0&xc\\0&yc\end{pmatrix}
  =cm.
\]
Hence $mA=\C m$ is one-dimensional.  Since $M$ is two-dimensional, no
single element generates it.
\end{example}

Thus finite generation does not imply idempotent generation even in a finite-dimensional Banach algebra.  For $\cB(E)$, proving that every finitely generated maximal right ideal is idempotent-generated would solve the stronger fixed-ideal problem by Proposition~\ref{prop:idempotent-BE}, but this requires a genuinely operator-theoretic argument.

\section{Scalar-plus-inessential spaces}

\begin{theorem}\label{thm:scalar-SS}
Let $E$ be an infinite-dimensional Banach space such that every operator
$T\in\cB(E)$ can be written as
\(
  T=\lambda\id_E+S
\)
for some $\lambda\in\C$ and $S\in\mathcal E(E)$. Then $\mathcal E(E)$ is the unique non-fixed maximal right ideal of $\cB(E)$.  If $\mathcal E(E)$ is finitely generated as a right ideal,
then there is a surjective inessential operator $E^n\to E$ for some
$n\in\N$.
\end{theorem}

\begin{proof}
Since $E$ is infinite-dimensional, $I_E$ is not inessential, so
$\mathcal E(E)$ is proper.  The scalar in the stated decomposition is
unique, and therefore $\mathcal E(E)$ has codimension one in
$\cB(E)$.  Hence it is a maximal right ideal.  Moreover,
$\cF(E)\subseteq\mathcal E(E)$, so the mutually exclusive
alternatives in Theorem~\ref{thm:A} show that $\mathcal E(E)$ is
non-fixed.  Every non-fixed maximal right ideal
contains $\mathcal E(E)$ by Corollary~\ref{cor:strictly-singular}, and
therefore equals it.

Suppose that $\mathcal E(E)$ is finitely generated.  By
Theorem~\ref{thm:B},
\(
  \mathcal E(E)=\Lift(T)
\)
for some $n\in\N$ and some surjection $T\in\cB(E^n,E)$.  Lemma~\ref{lem:row-operator-ideal}
shows that $T\in\mathcal E(E^n,E)$.
\end{proof}

The following corollary records the two most important instances in which this result applies. 

\begin{corollary}\label{cor:HI}${}$
\begin{enumerate}[label=\rm(\roman*)]
\item\label{cor:HI:i} Let $E$ be a hereditarily indecomposable Banach space. Then~$\cS(E)$ is the unique non-fixed maximal right ideal of $\cB(E)$.  Consequently, if there is no surjective strictly singular operator
$E^n\to E$ for any $n\in\N$, then $\cS(E)$ is not finitely
generated as a right ideal.  In particular, $\cS(E)$ is not finitely
generated when $E$ is reflexive.
\item\label{cor:HI:ii} Let $E$ be an infinite-dimensional Banach space which has `very few operators' in the sense that every operator on $E$ has the form $\lambda\id_E+K$ for some $\lambda\in\C$ and $K\in\cK(E)$.  Then $\cK(E)$ is the unique non-fixed maximal right ideal of $\cB(E)$, and it is not finitely generated.  

Consequently, $\cB(E)$ satisfies the right-sided Dales--\.{Z}elazko conjecture in this case.
\end{enumerate}
\end{corollary}

\begin{proof}
\ref{cor:HI:i}. Gowers and Maurey~\cite[Theorem~18]{GM:1993} have proved that every operator on a~hereditarily indecomposable Banach space~$E$ is a scalar multiple of the identity plus a strictly singular operator. (Note that at this point it is important that we work over the complex scalar field. Indeed, Ferenczi~\cite{Ferenczi:1997,Ferenczi:2007} has shown that for a real hereditarily indecomposable Banach space~$E$, $\cB(E)/\cS(E)$ is a division ring isomorphic to either~$\mathbb{R}$, $\C$, or the algebra~$\mathbb{H}$ of quaternions, and that all three possibilities occur.) 
Since 
\[ \cS(E)\subseteq\mathcal E(E)\subsetneq\cB(E), \] 
we deduce
that $\cS(E)=\mathcal E(E)$.  Lemma~\ref{lem:row-operator-ideal},
applied to both operator ideals, shows that an operator
$T\in\cB(E^n,E)$ belongs to either row ideal precisely when each
coordinate $T\iota_j$ belongs to the corresponding ideal on~$E$.
Consequently,
$\cS(E^n,E)=\mathcal E(E^n,E)$.  The first assertions now follow
from Theorem~\ref{thm:scalar-SS}.  

If $E$ is reflexive,
Theorem~\ref{thm:reflexive} supplies a maximal right ideal which is not finitely generated. It cannot be fixed, because fixed maximal right ideals are singly generated by Proposition~\ref{prop:fixed}; hence the
uniqueness just proved forces this ideal to be $\cS(E)$.

\ref{cor:HI:ii}. Arguing in the same way, just with the ideal of compact operators instead of the strictly singular operators, we obtain $\cK(E)=\mathcal E(E)$.  Applying Lemma~\ref{lem:row-operator-ideal} to the two operator ideals as above gives $\cK(E^n,E)=\mathcal E(E^n,E)$. Hence Theorem~\ref{thm:scalar-SS} shows that $\cK(E)$ is the unique
non-fixed maximal right ideal of $\cB(E)$.  If it were finitely
generated, Theorem~\ref{thm:scalar-SS} would give a surjective compact
operator $E^n\to E$ for some $n\in\N$, which is impossible because the range of a compact
operator cannot be both closed and infinite-dimensional. 
\end{proof}

The Argyros--Haydon space \cite{ArgyrosHaydon} is a prominent example of a Banach space which satisfies the scalar-plus-compact hypothesis of Corollary~\ref{cor:HI}\ref{cor:HI:ii}.

\section{The exact remaining problem}\label{sec:remaining}

Corollary~\ref{cor:exact-obstruction} gives a precise version of the stronger right-sided question for $\cB(E)$.

\begin{question}\label{q:strong}
Does there exist an infinite-dimensional Banach space $E$ which admits, for some $n\in\N$, a surjection $T\in\cB(E^n,E)$ which is not right invertible, but the row operator $[T\ S]$ given by~\eqref{thm:B:iv:row_op} is right invertible for every $S\in\cB(E)\setminus\Lift(T)$?
\end{question}

A negative answer for a given Banach space $E$ is equivalent to saying that every finitely generated maximal right ideal of $\cB(E)$ is fixed; this implies that the right-sided Dales--\.{Z}elazko conjecture is true for~$\cB(E)$ by Corollary~\ref{cor:fixed-implies-DZ}.  On the other hand, a positive answer would produce a non-fixed, finitely generated  maximal right ideal of~$\cB(E)$, but would \emph{not} disprove the right-sided Dales--\.{Z}elazko conjecture: $\cB(E)$ could still contain another maximal right ideal which is not finitely generated.

The reduction also identifies exactly why several natural arguments stop short.

First, a surjection $T\in\cB(E^n,E)$ need not be right invertible.
As Remark~\ref{rem:no-linear-selection} explains, the open mapping
theorem gives bounded pointwise choices, but not necessarily a bounded
linear choice.  Example~\ref{ex:mixed} shows that surjections need not
be right invertible even in simple separable spaces.  Maximality of the
right ideal~$\Lift(T)$ adds the completion property for the row operator
$[T\ S]$, but we have found no general argument which turns that
property into a right inverse for~$T$.

Second, a finitely generated maximal right ideal need not be generated
by an idempotent, as Example~\ref{ex:triangular} shows.
Theorem~\ref{thm:idempotent-criterion} is therefore only a conditional
criterion, not a proof of the conjecture.

Third, separability controls cardinality but is not by itself a counting
argument.  When $E$ is separable, $|\cB(E)|=\mathfrak{c}$, so
$\cB(E)$ contains only $\mathfrak{c}$ finite subsets.  One still needs
an independent reason for the existence of more than $\mathfrak{c}$
maximal right ideals.  The Boolean family in Theorem~\ref{thm:D}
supplies such a reason; separability of~$E$ alone does not.

Accordingly, the unresolved operator-theoretic issue is not
surjectivity, which is automatic for the operator attached to a non-fixed, 
finitely generated  maximal right ideal.  The issue is whether
a non-split quotient $E^n\twoheadrightarrow E$ can have a maximal
lifting ideal.  This is the point that requires further work.

\end{document}